\documentclass{article}
\usepackage[margin=1in]{geometry}

\AtEndDocument{\bigskip{\footnotesize%
  \textsc{$^*$Department of Applied Mathematics, Delhi Technological University, Delhi–110042, India} \par
  \textit{E-mail address:} \texttt{spkumar@dtu.ac.in} \par
  \addvspace{\medskipamount}
  \textsc{${}^{1}$Department of Applied Mathematics, Delhi Technological University, Delhi–110042, India} \par  
  \textit{E-mail address:} \texttt{suryagiri456@gmail.com} \par
}}
\usepackage{graphicx}
\usepackage{hyperref}
\usepackage[caption = false]{subfig}
\usepackage{amsthm}
\usepackage{amssymb}
\usepackage{mathrsfs}
\usepackage{mathtools}
\usepackage{enumerate}
\usepackage{amssymb}
\usepackage{wrapfig}
\usepackage{float}
\allowdisplaybreaks

\usepackage{thmtools}
\declaretheorem[numbered=no,
name=Theorem A]{theorem A}

\declaretheorem[numbered=no,
name=Theorem B]{theorem B}

\declaretheorem[numbered=no,
name=Theorem C]{theorem C}

\declaretheorem[numbered=no,
name=Theorem D]{theorem D}

\declaretheorem[numbered=no,
name=Theorem E]{theorem E}

\usepackage{geometry}
\numberwithin{equation}{section}
\DeclareMathOperator{\RE}{Re}

\theoremstyle{plain}
\usepackage{lipsum}

\newtheorem{theorem}{Theorem}[section]
\newtheorem{corollary}[theorem]{Corollary}

\newtheorem{lemma}{Lemma}[section]

\newtheorem{assumption}[theorem]{Assumption}
\theoremstyle{definition}
\newtheorem{definition}[theorem]{Definition}
\theoremstyle{remark}
\newtheorem{remark}{Remark}[section]

\makeatother

\usepackage{authblk}
\begin{document}
\title{Toeplitz determinants on bounded starlike circular domain in $\mathbb{C}^n$}
\author{Surya Giri$^{1}$ and S. Sivaprasad Kumar$^*$}


\date{}


	

\maketitle	
	
\begin{abstract}
  In this paper, we derive the sharp bounds of Toeplitz determinants for a class of holomorphic mappings on the bounded starlike circular domain $\Omega$ in $\mathbb{C}^n$, which extend certain known bounds for various subclasses of normalized analytic univalent functions in the unit disk to higher dimensions.
\end{abstract}

\vspace{0.5cm}
\maketitle

\section{Introduction}
   Let $\mathcal{S}$ be the class of normalized univalent holomorphic functions on the unit disk $\mathbb{U}$ in $\mathbb{C}$ of the form
   $$   g(z) = z + \sum_{n=2}^\infty b_n z^n . $$
   Let $\mathcal{S}^*$, $\mathcal{S}^*(\alpha)$ and $\mathcal{SS}^*(\beta)$ represent the subclasses of $\mathcal{S}$ containing the starlike functions, starlike functions of order $\alpha$ ($0\leq \alpha <1$) and strongly starlike functions of order $\beta$ ($0 < \beta \leq 1$), respectively (see \cite{8}).
   Ali et al. \cite{Ali} obtained the bounds of Toeplitz determinants formed over the coefficients of $g \in \mathcal{S}$.
     For $g(z) = z + \sum_{n=2}^\infty a_n z^n$, the Toeplitz matrix is given by
\begin{equation}
     T_{m,n}(g)= \begin{bmatrix}
	b_n & b_{n+1} & \cdots & b_{n+m-1} \\
	b_{n+1} & b_n & \cdots & b_{n+m-2}\\
	\vdots & \vdots & \vdots & \vdots\\
    b_{n+m-1} & b_{n+m-2} & \cdots & b_n\\
	\end{bmatrix}.
\end{equation}
   Thus, the second order Toeplitz determinant is
\begin{equation}
    \det{T_{2,2}(g)}=  b_2^2  -  b_3^2
\end{equation}
   and the third order Toeplitz determinant is given by
\begin{equation}\label{T31}
    \det T_{3,1}(g)= 2 b_2^{2} b_3 - 2 b_2^2-  b_3^2+1.
\end{equation}
   According to Ye and Lim \cite{LHLIM}, any $n \times n$ matrix over $\mathbb{C}$ generically can be written as the product of some Toeplitz matrices or Hankel matrices.
    Toeplitz matrices and Toeplitz determinants have many applications in pure as well as in applied mathematics \cite{Tpltz}.
    For more details of applications in various areas of mathematics, we refer \cite{LHLIM} and the references cited therein.

    For the class of normalized starlike functions, Ali et al. \cite{Ali} obtained the following result.
\begin{theorem A}\label{thmAH3}\cite{Ali}
   If $g \in \mathcal{S}^*$, then the following sharp bounds hold:
   $$ \vert \det T_{2,2}(g) \vert \leq 13 \;\;  \text{and} \;\; \vert \det T_{3,1}(g) \vert \leq  24.$$
\end{theorem A}
     Ahuja et al. \cite{Ahuja} determined the following estimates for the class of normalized starlike functions of order $\alpha$.
\begin{theorem B}\label{thmB}\cite{Ahuja}
   If $ g \in \mathcal{S}^*(\alpha)$, then
   $$ \vert \det T_{2,2}(g) \vert \leq (1 - \alpha)^2 ( 4 \alpha^2 - 12 \alpha + 13) $$
  and for $\alpha \in [0,2/3]$,
   $$ \vert \det T_{3,1}(g) \vert \leq 12 \alpha^4 -52 \alpha^3 + 91 \alpha^2 -74 \alpha + 24.$$
   All these estimates are sharp.
\end{theorem B}
  The following result directly follows from \cite[Theorem 1]{Ahuja} and \cite[Theorem 3]{Ahuja} for the class $\mathcal{SS}^*(\beta)$.
\begin{theorem C}
    If $g\in \mathcal{SS}^*(\beta)$, then for $\beta \in [1/3,1]$, the following sharp inequalities hold:
    $$\vert \det T_{2,2}(g) \vert  \leq 9 \beta^4 + 4 \beta^2 \quad \text{and} \quad \vert \det T_{3,1}(g) \vert \leq  15 \beta^4 + 8 \beta^2 + 1.  $$
\end{theorem C}
   Cartan \cite{GraKohH3} stated that the Bieberbach conjecture for the class $\mathcal{S}$ does not hold in case of several complex variables. There are various counterexamples, which show that many results in the Geometric function theory of one complex variable are not applicable for several complex variables (see \cite{GraKohH3}). Many researchers have focused on generalizing the coefficients inequalities for the subclasses of $\mathcal{S}$ in higher dimensions.
   Recently, Giri and Kumar~\cite{GiriH1} generalized the above results on the unit ball in a complex Banach space and on the unit polydisc in $\mathbb{C}^n.$ On the bounded starlike circular domain $\Omega \subset \mathbb{C}^n$, Liu and Xu~\cite{LiuXu3} (see also~\cite{XuH3}) solved the Fekete-Szeg\"{o} problem for a subclass of starlike mappings of order $\alpha$. Xu~\cite{Xu2H3} generalized the work in~\cite{XuH3} to a subclass of holomorphic mappings on the same domain $\Omega.$
   Contrary to the Fekete-Szeg\"{o} problems, very few results are known for the inequalities of homogeneous expansions for subclasses of biholomorphic mappings in several complex variables.  Results related to the bounds for the coefficients of various subclasses of holomorphic mappings in higher dimensions were obtained by Bracci et al.~\cite{Bracc}, Graham et al.~\cite{Grah2H3}, Graham et al.~\cite{GrahH3}, Graham et al.\cite{GrahLcH}, Hamada and Honda~\cite{Ham5H3}, Hamada et al.~\cite{Ham4H3}, Kohr~\cite{Kohr2H3}, Liu and Liu~\cite{LiuH3,Liu2H3}, Xu and Liu~\cite{Xu3H3,Xu4H4}.

   In this paper, we find the bounds of second and third order Toeplitz determinants for a class of holomorphic mappings on the bounded starlike circular domain in $\mathbb{C}^n$, which give an extension of Theorem  A, Theorem B and Theorem C to higher dimensions.
\section{Preliminaries}
     By $\mathbb{C}^n$, we denote the space of $n$ complex variables $z= (z_1, z_2,\cdots, z_n)'$ with the Euclidean inner product $\langle z,w \rangle = \sum_{i=1}^n z_i \overline{w}_i$ and the norm $\| z \| =\langle z,w \rangle^{1/2}$. Let $\mathbb{U}^n$ be the Euclidean unit ball in $\mathbb{C}^n$ and $\Omega \subset \mathbb{C}^n$ be a bounded starlike circular domain with $0 \in \Omega$ and its Minkowski functional $\rho(z) \in C^{1}$ in $\mathbb{C}^n \setminus\{0\}.$ Let $\mathcal{H}(\Omega)$ be the set of all holomorphic mappings from $\Omega$ into $\mathbb{C}^n$.  If $g\in \mathcal{H}(\Omega)$, then
    $$ g(w) = \sum_{k=0}^\infty \frac{1}{k!} D^k g(z) ((w -z)^k )$$
    for all $w$ in some neighborhood of $z$, where $D^k g(z)$ is the $k$th Fr\'{e}chet derivative of $g$ at $z$. A function $g\in \mathcal{H}(\Omega)$ is said to be biholomorphic if $g(\Omega)$ is a domain in $\mathbb{C}^n$ and inverse of $g$ exists, which is holomorphic on $g(\Omega).$ Let $J_g(z)$ be the Jacobian matrix of $g$ and $\det J_g(z)$ be the Jacobian determinant of $g$ at $z\in \Omega$. A mapping $g \in \mathcal{H}(\Omega)$ is said to be locally biholomorphic if $\det J_g(z) \neq 0$ for all $z\in \Omega$. In higher dimensions, $g\in \mathcal{H}(\Omega)$ is said to be normalized if $g(0)=0$ and $J_g(0)=I$, where $I$ is the identity matrix. Let $\mathcal{S}^*(\Omega)$ denotes the class of starlike mappings on $\Omega$. When $n=1$, $\Omega = \mathbb{U}$, the class $\mathcal{S}^*(\mathbb{U})$ is denoted by $\mathcal{S}^*.$

    On a bounded circular domain $\Omega \subset \mathbb{C}^n$, the first and the $m^{th}$ Fr\'{e}chet derivative of a holomorphic mapping $g \in \mathcal{H}(\Omega)$  are written by
    $ D g(z)$ and $D^m g(z) (a^{m-1},\cdot)$, respectively. The matrix representations are
\begin{align*}
    D g(z) &= \bigg(\frac{\partial g_j}{\partial z_k} \bigg)_{1 \leq j, k \leq n}, \\
    D^m g(z)(a^{m-1}, \cdot) &= \bigg( \sum_{p_1,p_2, \cdots, p_{m-1}=1}^n  \frac{ \partial^m g_j (z)}{\partial z_k \partial z_{p_1} \cdots \partial z_{p_{m-1}}} a_{p_1} \cdots a_{p_{m-1}}   \bigg)_{1 \leq j,k \leq n},
\end{align*}
   where $g(z) = (g_1(z), g_2(z), \cdots g_n(z))', a= (a_1, a_2, \cdots a_n)'\in \mathbb{C}^n.$

    According to Liu and Lu~\cite{LiudH3} (see also \cite{Ham2H3}), we have the following definition.
\begin{definition}
     Let $\Omega$ be a bounded starlike circular domain in $\mathbb{C}^n$ with $0 \in \Omega$ and its Minkowski functional $\rho \in C^{1}$ in $\mathbb{C}^n \setminus\{0\}.$ A normalized locally biholomorphic mapping $g : \Omega \rightarrow \mathbb{C}^n$ is said to be starlike of order $\alpha$ $(0 \leq \alpha <1)$ if
     $$  \bigg\vert \frac{2}{\rho(z)} \frac{\partial \rho}{\partial z} J_g^{-1}(z) g(z) - \frac{1}{2\alpha}  \bigg\vert < \frac{1}{2\alpha}, \quad \forall z\in \Omega \setminus\{0\}.    $$
     Equivalently, the above equation can be written as
      $$ \RE \bigg\{ \frac{\rho(z)}{ 2 \frac{ \partial \rho(z)}{ \partial z} J_g^{-1}(z) g(z) }\bigg\} > \alpha , \quad \forall z\in \Omega \setminus\{0\}. $$
    Clearly, when $\Omega = \mathbb{U}^n$, the  aforementioned inequality is equivalent to
    $$ \RE \bigg\{ \frac{\|z \|^2}{ \langle J_g^{-1}(z) g(z),z \rangle }\bigg\} > \alpha , \quad \forall z\in \mathbb{U}^n \setminus\{0\}. $$
    In case of $n=1$, $\Omega = \mathbb{U}$ and the above relation is equivalent to
    $$ \RE \frac{z g'(z) }{g(z)} >0 , \quad z \in \mathbb{U}.$$
    We denote by $\mathcal{S}^*_{\alpha}(\Omega)$ the set of all starlike mappings of order $\alpha$ on $\Omega.$
\end{definition}
\begin{definition}\cite{GraKohH3}(also see~\cite{KohrH3,Ham3H3})
     Let $\Omega$ be a bounded starlike circular domain in $\mathbb{C}^n$ with $0 \in \Omega$ and its Minkowski functional $\rho \in C^{1}$ in $\mathbb{C}^n \setminus\{0\}.$ A normalized locally biholomorphic mapping $g : \Omega \rightarrow \mathbb{C}^n$ is said to be strongly starlike of order $\beta$ $(0 < \beta \leq 1)$ if
     $$  \bigg\vert \arg \frac{2}{\rho(z)} \frac{\partial \rho}{\partial z} J_g^{-1}(z) g(z)  \bigg\vert < \frac{\pi}{2} \beta, \quad \forall z\in \Omega \setminus\{0\}.    $$
    Clearly, when $\Omega = \mathbb{U}^n$, the  aforementioned inequality is equivalent to
    $$ \vert \arg \langle J_g^{-1}(z) g(z),z \rangle \vert  < \frac{\pi}{2}\beta , \quad \forall z\in \mathbb{U}^n \setminus\{0\}. $$
    In case of $n=1$, $\Omega = \mathbb{U}$ and the above relation is equivalent to
    $$ \bigg\vert \arg \frac{z g'(z) }{g(z)} \bigg\vert < \frac{\pi}{2}\beta , \quad \forall z \in \mathbb{U}.$$
    We denote by $\mathcal{SS}^*_{\beta}(\Omega)$ the set of all strongly starlike mappings of order $\beta$ on $\Omega.$
\end{definition}
   Next, we recall the class $\mathscr{M}$, which plays a fundamental role in the study of Loewner chains and Loewner differential equation in several complex variables (see~\cite{GraKohH3,Pfa2H3}.
\begin{equation*}
   \mathscr{M}= \left\{ p \in \mathcal{H}(\Omega) : p(0) =0, J_p(0)=I, \RE  \frac{\partial\rho}{\partial z} p(z)  >0, z\in \Omega\setminus \{ 0 \} \right\},
\end{equation*}
    where
    $\partial \rho(z)/\partial z = (\partial \rho(z)/\partial z_1, \partial \rho(z)/\partial z_2, \cdots,\partial \rho(z)/\partial z_n).$

    Kohr~\cite{Kohr2H3} introduced the class $\mathscr{M}_\Phi$ on $\mathbb{U}^n$, which is studied by Graham et al.~\cite{GrahH3} (see also~\cite{Grah2H3}), where $\Phi : \mathbb{U} \rightarrow \mathbb{C}$ is a biholomorphic function such that $\Phi(0)=1$ and $\RE \Phi(z)>0$ on $\mathbb{U}.$ Recently, Xu et al.~\cite{XuH3} considered the class $\mathscr{M}_\Phi$ on $\Omega \subset \mathbb{C}^n$.  Here, we add some more conditions on $\Phi$ and define the following subsets of $\mathscr{M}.$
\begin{assumption}
    Let $\Phi : \mathbb{U}\rightarrow \mathbb{C}$ be a biholomorphic function such that $\Phi(0)=1$, $\Phi'(0)>0$, $\Phi''(0)\in \mathbb{R}$ and $\RE \Phi(z) >0$ on $\mathbb{U}.$
\end{assumption}
\noindent        Obviously, there are many functions which satisfy this assumption. Let
\begin{equation*}
   \mathscr{M}_\Phi = \left\{ p \in \mathcal{H}(\Omega) : p(0) =0, J_p(0)=I, \frac{\rho(z)}{ 2 \frac{\partial\rho}{\partial z} p(z) } \in \Phi(\mathbb{U}), z\in \Omega\setminus \{ 0 \} \right\}.
\end{equation*}
     The class $\mathscr{M}_\Phi$ coincides with $\mathscr{M}$ for $\Phi(z) = (1+z)/(1-z)$, $z\in \mathbb{U}$.  Also, if $\Omega =\mathbb{U}^n$, then
\begin{equation*}
   \mathscr{M}_\Phi = \left\{ p \in \mathcal{H}(\mathbb{U}^n) : p(0) =0, J_p(0)=I, \frac{\| z \|^2}{ \langle p(z),z \rangle } \in \Phi(\mathbb{U}), z\in \mathbb{U}^n\setminus \{ 0 \} \right\}.
\end{equation*}
\begin{remark}\label{rem1}
     Let $g\in \mathcal{H}(\mathbb{U})$ be a normalized locally biholomorphic function.  If $J_g^{-1}(z) g(z) \in \mathcal{M}_\Phi$, then for different choices of $\Phi$, we obtain different important classes of $\mathcal{S}(\Omega)$. For instance, if we take $\Phi(z) = (1+z)/(1-z)$, $\Phi(z) =  (1+(1 - 2\alpha )z)/(1-z)$ and  $\Phi(z) = ( (1+z)/(1-z))^{\beta}$ (where the branch point is chosen such that $( (1+z)/(1-z))^{\beta} =1$ at $z=0$), then we  easily obtain $g \in \mathcal{S}^*(\Omega)$, $g \in \mathcal{S}^*_\alpha(\Omega)$ and $g \in \mathcal{SS}^*_\beta(\Omega)$, respectively.
\end{remark}
   The following lemma helps us to prove the main results.
\begin{lemma}\cite{LiugH3}
     $\Omega \subset \mathbb{C}^n$ is a bounded starlike circular domain if and only if there exists a unique real continuous functions $\rho: \mathbb{C}^n \rightarrow \mathbb{R}$, called the Minkowski functional of $\Omega$, such that
\begin{enumerate}[(i)]
  \item $\rho(z) \geq 0$, $z \in \mathbb{C}^n$; $\rho(z) = 0  \Leftrightarrow z=0$;
  \item $\rho(t z) = \vert t \vert \rho(z)$, $t\in \mathbb{C}$, $z \in \mathbb{C}^n$;
  \item $\Omega = \{ z \in \mathbb{C}^n : \rho(z)< 1\}.$
\end{enumerate}
  Furthermore, if $\rho(z) \in C^{1}$ in $\mathbb{C}^n\setminus \{0\}$, then the function $\rho(z)$ has the following properties.
\begin{equation}\label{eqnmf}
      2 \frac{\partial \rho(z)}{\partial z} z = \rho(z), \quad z \in \mathbb{C}^n,
\end{equation}
   $$ 2 \frac{\partial \rho(z_0)}{\partial z} z_0 = 1, \quad z_0 \in \partial \Omega, $$
   $$ \frac{\partial \rho (\lambda z)}{\partial z} =  \frac{\partial \rho (z)}{\partial z} , \quad \lambda \in (0, \infty),     $$
   $$ \frac{\partial \rho (e^{i\theta} z)}{\partial z} = e^{-i \theta} \frac{\partial \rho (z)}{\partial z} , \quad \theta \in\mathbb{R}.    $$
\end{lemma}
\section{Main Results}
     The sharp bounds of second and third order Toeplitz determinants for a class of holomorphic mappings on $\Omega$ are derived in this section. Later, applications of these results for other interesting subclasses of $\mathcal{S}(\Omega)$ are given.
\begin{theorem}\label{thm1H3}
    Let $g\in \mathcal{H}(\Omega,\mathbb{C})$ with $g(0)=1$ and $G(z) = z g(z)$. If $J_G^{-1}(z) G(z) \in \mathscr{M}_\Phi$ such that $\Phi$ satisfies
    $$ \vert \Phi''(0) + 2 (\Phi'(0))^2 \vert  \geq 2  \Phi'(0) > 0, $$
    then
\begin{align*}
    \bigg\vert    \bigg( 2 \frac{\partial \rho}{\partial z}  \frac{ D^3 G(0) (z^3)}{3!  \rho^3 (z)} \bigg)^2 - & \bigg(2 \frac{\partial \rho}{\partial z} \frac{D^2 G(0) (z^2)}{2! \rho^2 (z)} \bigg)^2  \bigg\vert
   \leq ( \Phi'(0) )^2 +  \frac{( \Phi'(0))^2}{4}   \left( \frac{1}{2} \frac{\Phi''(0)}{\Phi'(0)} + \Phi'(0) \right)^2.
\end{align*}
    The bound is sharp.
\end{theorem}
\begin{proof}
    Since $J^{-1}_G(z)$ exists, therefore $g(z) \neq 0$, $z \in \Omega$. For fix $z \in \Omega \setminus \{ 0\}$, let us denote $z_0 = \frac{z}{\rho(z)}$ and define $h: \mathbb{U} \rightarrow \mathbb{C}$ such that
\begin{equation*}
    h(\zeta) = \left\{ \begin{array}{ll}
     \frac{\zeta}{ 2 \frac{\partial \rho (z_0)}{\partial z} J^{-1}_G(\zeta z_0) G( \zeta z_0)}, & \;\;\zeta \neq 0, \\
    1, & \;\; \zeta =0.
    \end{array}
    \right.
\end{equation*}
    Using the property  $2 \frac{\partial \rho(z_0)}{\partial z}z_0 =1$ for $z_0 \in \partial \Omega$ of Minkowski functional, we obtain $h \in \mathcal{H}(\mathbb{U})$ and since $J^{-1}_G( z) G( z) \in \mathcal{M}_\Phi$, therefore
\begin{align*}
     h(\zeta) &=     \frac{\zeta}{ 2 \frac{\partial \rho (z_0)}{\partial z} J^{-1}_G(\zeta z_0) G( \zeta z_0)} \\
             &= \frac{\rho( \zeta  z_0)}{ 2 \frac{\partial \rho (\zeta  z_0)}{\partial z} J^{-1}_G(\zeta z_0) G( \zeta z_0)} \in \Phi (\mathbb{U}), \quad \zeta \in \mathbb{U}\setminus\{ 0 \}.
\end{align*}
    Applying the same technique as in \cite{PfaH3} (also see \cite[Theorem 7.1.14]{GraKohH3}), we obtain
\begin{equation*}
     J^{-1}_G (z) =   \frac{1}{g(z)} \bigg( I - \frac{ \frac{z J_g(z)}{g(z)}}{1 + \frac{J_g(z) z}{g(z)}} \bigg).
\end{equation*}
     Now, using $G(z) = z g(z)$, we have
     $$   J^{-1}_G (z) G(z) =   \frac{z g(z)}{ g(z) + J_g (z) z}, \quad z \in \Omega \setminus\{0\},  $$
   which together with (\ref{eqnmf}) gives
   $$  \frac{\rho(z)}{2 \frac{\partial \rho(z)}{\partial z} J^{-1}_G (z) G(z) }  =  1 + \frac{J_g (z) z}{g(z)} , \quad  z \in \Omega \setminus \{0 \}. $$
   In view of the above equation, we obtain
   $$  h(\zeta) = \frac{\rho(\zeta z_0)}{2 \frac{\partial \rho(\zeta z_0)}{\partial z} J^{-1}_G (\zeta z_0) G(\zeta z_0) }  =  1 + \frac{J_g (\zeta z_0) \zeta z_0}{g(\zeta z_0)},   $$
   which immediately yields
   $$  h(\zeta) g(\zeta z_0) = g(\zeta z_0) +  J_g (\zeta z_0) \zeta z_0 . $$
   Based on the Taylor series expansions in $\zeta$, the above equation gives
\begin{align*}
   \bigg( 1 &+ h'(0) \zeta + \frac{h''(0)}{2!} \zeta^2 + \cdots \bigg) \bigg( 1 + J_g(0) (z_0) \zeta + \frac{D^2 g(0) (z_0^2)}{2!} \zeta^2 + \cdots \bigg)   \\
  &= \bigg( 1 + J_g (0) (z_0) \zeta + \frac{D^2 g(0) (z_0^2)}{2!} \zeta^2 + \cdots  \bigg) \bigg( J_g(0) (z_0) \zeta + D^2 g(0) (z_0^2) \zeta^2 + \cdots \bigg).
\end{align*}
  By the comparison of homogeneous expansions, we get
   $$ h'(0) = J_g(0) (z_0) . $$
   Further, using $z_0 = \frac{z}{\rho(z)}$ in the above relation, we have
\begin{equation}\label{eq1H3}
     h'(0) \rho(z) = J_g(0) (z).
\end{equation}
   Since, we also have $G(z) = z g(z)$, therefore
   $$ \frac{D^2 G(0) z^2}{2!} = J_g(0)(z) z, $$
  which together with (\ref{eqnmf}) leads to
\begin{equation}\label{eqnn}
     2 \frac{\partial \rho}{\partial z} \frac{D^2 G(0) z^2}{2!} = J_g(0)(z) \rho(z).
\end{equation}
    Thus, from (\ref{eq1H3}) and (\ref{eqnn}), we obtain
\begin{equation*}
     2 \frac{\partial \rho}{\partial z} \frac{D^2 G(0) z^2}{2! \rho^2(z)} = h'(0) .
\end{equation*}
    Since $h \prec \Phi$, therefore $\vert h'(0) \vert \leq \vert \Phi'(0) \vert$, using this fact, we get
\begin{equation}\label{a2H3}
   \bigg\vert  2 \frac{\partial \rho}{\partial z} \frac{D^2 G(0) z^2}{2! \rho^2(z)} \bigg\vert \leq \vert \Phi'(0) \vert .
\end{equation}
   For $\lambda \in \mathbb{C}$, Xu et al. \cite[Theorem 1]{XuH3} proved that
\begin{equation}\label{FSBH3}
\begin{aligned}
\left.
\begin{array}{ll}
   &\bigg\vert   2 \dfrac{\partial \rho}{\partial z} \dfrac{ D^3 G(0) (z^3)}{3!  \rho^3 (z)} -  \lambda \bigg(2 \dfrac{\partial \rho}{\partial z} \dfrac{D^2 G(0) (z^2)}{2! \rho^2 (z)} \bigg)^2 \bigg\vert \\
   &\quad \quad \quad \quad \quad  \leq \dfrac{\vert \Phi'(0) \vert}{2} \max \left\{ 1,  \left\lvert \dfrac{1}{2} \dfrac{\Phi''(0)}{\Phi'(0)} +  (1 -  2 \lambda )  \Phi'(0) \right\rvert \right\} , \;\;  z \in \Omega \setminus \{0 \}.
\end{array}
\right\}
\end{aligned}
\end{equation}
    Thus, when $  \vert \Phi''(0) + 2 (\Phi'(0))^2 \vert  \geq 2  \Phi'(0) $, the equation (\ref{FSBH3}) readily yields
\begin{equation}\label{a3H3}
   \bigg\vert   2 \frac{\partial \rho}{\partial z} \frac{ D^3 G(0) (z^3)}{3!  \rho^3 (z)} \bigg\vert \leq  \frac{ \Phi'(0)}{2}   \left\lvert \frac{1}{2} \frac{\Phi''(0)}{\Phi'(0)} + \Phi'(0) \right\rvert.
\end{equation}
   Using the bounds given in (\ref{a2H3}) and (\ref{a3H3}), together with the following inequality
\begin{align*}
   \bigg\vert    \bigg( 2 \frac{\partial \rho}{\partial z}  \frac{ D^3 G(0) (z^3)}{3!  \rho^3 (z)} \bigg)^2 - & \bigg(2 \frac{\partial \rho}{\partial z} \frac{D^2 G(0) (z^2)}{2! \rho^2 (z)} \bigg)^2  \bigg\vert &  \\
   &\quad \quad  \quad \quad \leq    \bigg\vert 2 \frac{\partial \rho}{\partial z}  \frac{ D^3 G(0) (z^3)}{3!  \rho^3 (z)} \bigg\vert^2 +  \bigg\vert 2 \frac{\partial \rho}{\partial z} \frac{D^2 G(0) (z^2)}{2! \rho^2 (z)} \bigg\vert^2,
\end{align*}
    we find the required bound.

   To see the sharpness of the bound consider the function
\begin{equation}\label{extHB}
    G(z) = z \exp \int_0^{\frac{z_1}{r}} \frac{( \Phi(i t)-1) }{t}dt, \quad z\in \Omega,
\end{equation}
   where $ r = \sup \{ \vert z_1 \vert : z = (z_1, z_2, \cdots, z_n)' \in \Omega    \}.$ It can be easily showed that $J^{-1}_G (z) G(z) \in \mathcal{M}_\Phi$ and
    $$  \frac{D^2  G(0) (z^2)}{2!}= i \Phi'(0) (\frac{z_1}{r}) z \;\; \text{and} \;\; \frac{D^3  G(0) (z^3)}{3!} = - \frac{1}{2} \left( \frac{\Phi''(0)}{2} + (\Phi'(0))^2 \right) (\frac{z_1}{r})^2 z. $$
    By applying (\ref{eqnmf}) in the above relations, we get
\begin{align*}
    2 \frac{\partial \rho}{\partial z}   \frac{D^2  G(0) (z^2) }{2!} &= i \Phi'(0) (\frac{z_1}{r}) \rho(z)
\end{align*}
    and
\begin{align*}
      2 \frac{\partial \rho}{\partial z} \frac{D^3  G(0) (z^3)}{3!} \rho(z) &= - \frac{1}{2} \left( \frac{\Phi''(0)}{2} + (\Phi'(0))^2 \right) (\frac{z_1}{r})^2 \rho^2(z).
\end{align*}
    Setting $z = R u$ $(0< R <1)$, where $u =(u_1, u_2, \cdots, u_n)' \in \partial \Omega$ and $u_1 =r$, we obtain
\begin{align}\label{b2H3}
    2 \frac{\partial \rho}{\partial z}   \frac{D^2  G(0) (z^2) }{2! \rho^2(z)} &= i \Phi'(0)
\end{align}
   and
\begin{align}\label{b3H3}
      2 \frac{\partial \rho}{\partial z} \frac{D^3  G(0) (z^3)}{3! \rho^3(z)}  &= - \frac{1}{2} \left( \frac{\Phi''(0)}{2} + (\Phi'(0))^2 \right).
\end{align}
   Thus, from (\ref{b2H3}) and (\ref{b3H3}), we have
\begin{align*}
     & \bigg\vert   \bigg( 2 \frac{\partial \rho}{\partial z}  \frac{ D^3 G(0) (z^3)}{3!  \rho^3 (z)} \bigg)^2 -  \bigg(2 \frac{\partial \rho}{\partial z} \frac{D^2 G(0) (z^2)}{2! \rho^2 (z)} \bigg)^2  \bigg\vert    \\
    & \quad \quad \quad \quad \quad \quad \quad \quad \quad\quad \quad \quad \quad \quad \quad\quad \leq   \frac{( \Phi'(0))^2}{4} \left( \frac{1}{2} \frac{\Phi''(0)}{\Phi'(0)} + \Phi'(0) \right)^2 + ( \Phi'(0) )^2 ,
\end{align*}
    which shows the sharpness of the bound and  completes the proof.
\end{proof}
\begin{theorem}\label{thm2H3}
    Let $g\in \mathcal{H}(\Omega,\mathbb{C})$ with $g(0)=1$ and $G(z) = z g(z)$. If $J^{-1}_G (z) G(z) \in \mathscr{M}_\Phi$ such that $\Phi$ satisfies
    $$ 2 \Phi'(0) - 2 (\Phi'(0))^2 \leq {\Phi''(0)} \leq 6 (\Phi'(0))^2 - 2 \Phi'(0) ,$$
    then
\begin{align*}
     \vert 2 b_2^2 b_3 & - b_3^2 - 2 b_2^2 + 1 \vert   \\
     &  \quad \quad \quad  \quad \quad \leq 1 + 2 ( \Phi'(0) )^2 + \frac{ ( \Phi'(0) )^2}{4} \bigg( 3 \Phi'(0)  - \frac{\Phi''(0)}{2 \Phi'(0)} \bigg)  \bigg( \frac{\Phi''(0)}{2 \Phi'(0)} + \Phi'(0) \bigg),
\end{align*}
   where
\begin{align}\label{b2b3}
       b_3 = 2 \frac{\partial \rho}{\partial z} \frac{ 2 D^3 G(0) (z^3)}{3! \rho^3(z) }   \;\; \text{and} \;\; b_2 =  2 \frac{\partial \rho}{\partial z} \frac{ 2 D^2 G(0) (z^2)}{2! \rho^2(z) }.
\end{align}
   The bound is sharp.
\end{theorem}
\begin{proof}
    Since $2 \Phi'(0) <  \Phi''(0) + 2 (\Phi'(0))^2$, the inequality (\ref{FSBH3}) gives
\begin{equation}\label{eqnthm2}
    \bigg\vert   2 \frac{\partial \rho}{\partial z} \frac{ D^3 G(0) (z^3)}{3!  \rho^3 (z)} \bigg\vert \leq  \frac{ \Phi'(0)}{2}   \left( \frac{1}{2} \frac{\Phi''(0)}{\Phi'(0)} + \Phi'(0) \right).
\end{equation}
     Also, since  $ 2 \Phi'(0)  + \Phi''(0) \leq 6 (\Phi'(0))^2 ,$ the inequality (\ref{FSBH3}) for $\lambda =2$ gives
\begin{equation}\label{FS2H3}
     \bigg\vert 2 \frac{\partial \rho}{\partial z} \frac{D^3 G(0) (z^3)}{3!  \rho^3(z)} - 2 \bigg( 2 \frac{\partial \rho}{\partial z}  \frac{D^2 G(0) (z^2)}{2! \rho^2(z)} \bigg)^2 \bigg\vert \leq  \frac{\Phi'(0) }{2}   \bigg( 3  \Phi'(0)  - \frac{1}{2} \frac{\Phi''(0)}{\Phi'(0)} \bigg).
\end{equation}
    Using the estimates given in (\ref{a2H3}) and (\ref{eqnthm2}), and the bound given by (\ref{FS2H3}) in the following inequality
\begin{equation*}\label{T31B}
    \vert 2 b_2^2 b_3  - b_3^2 - 2 b_2^2 + 1 \vert \leq  1 + 2 \vert b_2 \vert^2 + \vert b_3\vert \vert b_3 - 2 b_2^2\vert
\end{equation*}
  the required bound is established.

  The result is sharp for the function $G(z)$ given by (\ref{extHB}). As for this function, we have
    $b_2 = i \Phi'(0)$ and
   $ b_3  = -( \Phi''(0) + 2(\Phi'(0))^2 )/4$  from (\ref{b2H3}) and (\ref{b3H3}), respectively. Therefore
   $$   1 - b_3 ( b_3 - 2 b_2^2 ) - 2 b_2^2  = 1 + 2 ( \Phi'(0) )^2 + \frac{ ( \Phi'(0) )^2}{4} \bigg( 3 \Phi'(0)  - \frac{\Phi''(0)}{2 \Phi'(0)} \bigg)  \bigg( \frac{\Phi''(0)}{2 \Phi'(0)} + \Phi'(0) \bigg),$$
    which proves the sharpness of the bound.
\end{proof}
   In case of $\Omega = \mathbb{U}^n$, Theorem \ref{thm1H3} and Theorem \ref{thm2H3} directly give the following results, which we state here without proof.
\begin{theorem}
   Let $g\in \mathcal{H}(\mathbb{U}^n, \mathbb{C})$ with $g(0)=1$ and $G(z)= z g(z)$. If $J^{-1}_G(z) G(z) \in \mathscr{M}_\Phi$ such that $\Phi$ satisfies
    $$ \vert \Phi''(0) + 2 (\Phi'(0))^2 \vert  \geq 2  \Phi'(0) > 0, $$
    then
\begin{align*}
   & \bigg\vert    \bigg( \frac{1}{\| z \|^4}  \frac{ D^3 G(0) (z^3)}{3! } \bar{z} \bigg)^2 -    \bigg(\frac{1}{\| z \|^3} \frac{D^2 G(0) (z^2)}{2! } \bar{z} \bigg)^2  \bigg\vert \\
    & \quad \quad \quad \quad \quad \quad \quad \quad \quad \quad \quad \quad \quad \quad \quad \leq ( \Phi'(0) )^2 +  \frac{( \Phi'(0))^2}{4}   \left( \frac{1}{2} \frac{\Phi''(0)}{\Phi'(0)} + \Phi'(0) \right)^2.
\end{align*}
   The bound is sharp.
\end{theorem}
\begin{theorem}\label{thm4H3}
    Let $g\in \mathcal{H}(\mathbb{U}^n,\mathbb{C})$ with $g(0)=1$ and $G(z) = z g(z)$. If $J^{-1}_G (z) G(z) \in \mathscr{M}_\Phi$ such that $\Phi$ satisfies
    $$ 2 \Phi'(0) - 2 (\Phi'(0))^2 \leq {\Phi''(0)} \leq 6 (\Phi'(0))^2 - 2 \Phi'(0) ,$$
    then
\begin{align*}
     &\vert 2 d_2^2 d_3  - d_3^2 - 2 d_2^2 + 1 \vert   \\
      &\quad\quad\quad \quad\quad \leq  \frac{ ( \Phi'(0) )^2}{4} \bigg( 3 \Phi'(0)  - \frac{\Phi''(0)}{2 \Phi'(0)} \bigg)  \bigg( \frac{\Phi''(0)}{2 \Phi'(0)} + \Phi'(0) \bigg) + 2 ( \Phi'(0) )^2 +1,
\end{align*}
   where
\begin{align}\label{d2d3}
       d_3 =  \frac{1}{\| z \|^4}  \frac{ D^3 G(0) (z^3)}{3! } \bar{z}   \;\; \text{and} \;\; d_2 =  \frac{1}{\| z \|^3} \frac{D^2 G(0) (z^2)}{2! } \bar{z}.
\end{align}
   The bound is sharp.
\end{theorem}
    In sight of remark \ref{rem1}, various choices of $\Phi$ in Theorem \ref{thm1H3} to Theorem \ref{thm4H3} lead to the following results for different subclasses of $\mathcal{S}(\Omega)$.
\begin{corollary}
    If $g: \Omega \rightarrow \mathbb{C}$ and $G(z) = z g(z) \in \mathcal{S}^*(\Omega)$, then
\begin{align*}
    \bigg\vert    \bigg( 2 \frac{\partial \rho}{\partial z}  \frac{ D^3 G(0) (z^3)}{3!  \rho^3 (z)} \bigg)^2 - & \bigg(2 \frac{\partial \rho}{\partial z} \frac{D^2 G(0) (z^2)}{2! \rho^2 (z)} \bigg)^2  \bigg\vert
   \leq  13
\end{align*}
   and
\begin{align*}
     \vert 2 b_2^2 b_3 & - b_3^2 - 2 b_2^2 + 1 \vert  \leq 24,
\end{align*}
   where $b_2$ and $b_3$ are given by (\ref{b2b3}).
      All these estimations are sharp.
\end{corollary}
\begin{corollary}
    If $g: \Omega \rightarrow \mathbb{C}$ and $G(z) = z g(z) \in \mathcal{S}^*_\alpha(\Omega)$, then
\begin{align*}
    \bigg\vert    \bigg( 2 \frac{\partial \rho}{\partial z}  \frac{ D^3 G(0) (z^3)}{3!  \rho^3 (z)} \bigg)^2 - & \bigg(2 \frac{\partial \rho}{\partial z} \frac{D^2 G(0) (z^2)}{2! \rho^2 (z)} \bigg)^2  \bigg\vert
   \leq (1 - \alpha)^2 ( 4 \alpha^2 - 12 \alpha + 13)
\end{align*}
   and for $\alpha \in [0,2/3]$,
\begin{align*}
     \vert 2 b_2^2 b_3 & - b_3^2 - 2 b_2^2 + 1 \vert \leq 12 \alpha^4 -52 \alpha^3 + 91 \alpha^2 -74 \alpha + 24,
\end{align*}
   where $b_2$ and $b_3$ are given by (\ref{b2b3}). All these estimations are sharp.
\end{corollary}
\begin{corollary}
    If $g: \Omega \rightarrow \mathbb{C}$ and $G(z) = z g(z) \in \mathcal{SS}^*_\beta(\Omega)$, then for $\beta \in [1/3,1]$, the following sharp inequalities hold:
\begin{align*}
    \bigg\vert    \bigg( 2 \frac{\partial \rho}{\partial z}  \frac{ D^3 G(0) (z^3)}{3!  \rho^3 (z)} \bigg)^2 - & \bigg(2 \frac{\partial \rho}{\partial z} \frac{D^2 G(0) (z^2)}{2! \rho^2 (z)} \bigg)^2  \bigg\vert
   \leq 9 \beta^4 + 4 \beta^2
\end{align*}
   and
\begin{align*}
     \vert 2 b_2^2 b_3 & - b_3^2 - 2 b_2^2 + 1 \vert \leq 15 \beta^4 + 8 \beta^2 + 1,
\end{align*}
   where $b_2$ and $b_3$ are given by (\ref{b2b3}).
\end{corollary}
   If $\Omega = \mathbb{U}^n$, we obtain the following bounds.
\begin{corollary}
 If $g: \mathbb{U}^n \rightarrow \mathbb{C}$ and $G(z) = z g(z) \in \mathcal{S}^*(\mathbb{U}^n)$, then
\begin{align}\label{crl1H3}
   \bigg\vert    \bigg( \frac{1}{\| z \|^4}  \frac{ D^3 G(0) (z^3)}{3! } \bar{z} \bigg)^2 -    \bigg(\frac{1}{\| z \|^3} \frac{D^2 G(0) (z^2)}{2! } \bar{z} \bigg)^2  \bigg\vert
   \leq 13
\end{align}
   and
\begin{align}\label{crl2H3}
     \vert 2 d_2^2 d_3 & - d_3^2 - 2 d_2^2 + 1 \vert \leq 24,
\end{align}
   where  $d_2$ and $d_3$ are given by (\ref{d2d3}).
   All these bounds are sharp.
\end{corollary}
\begin{remark}
     When $n=1$, (\ref{crl1H3}) and (\ref{crl2H3}) reduce to the following:
\begin{align*}
    \bigg\vert \bigg( \frac{G^{(3)}(0)}{3!} \bigg)^2 - \bigg( \frac{G''{(0)}}{2!}  \bigg)^2  \bigg\vert \leq 13
\end{align*}
   and
\begin{align*}
     \vert 2 d_2^2 d_3 & - d_3^2 - 2 d_2^2 + 1 \vert \leq 24,
\end{align*}
   where
\begin{align*}
       d_3 =   \frac{  G^{(3)}(0)}{3! }    \;\; \text{and} \;\; d_2 =  \frac{ G''(0) }{2! } .
\end{align*}
    which are equivalent to the bounds given in Theorem A.
\end{remark}
\begin{corollary}
 If $g: \mathbb{U}^n \rightarrow \mathbb{C}$ and $G(z) = z g(z) \in \mathcal{S}^*_\alpha(\mathbb{U}^n)$, then
\begin{align}\label{crl3H3}
   \bigg\vert    \bigg( \frac{1}{\| z \|^4}  \frac{ D^3 G(0) (z^3)}{3! } \bar{z} \bigg)^2 -    \bigg(\frac{1}{\| z \|^3} \frac{D^2 G(0) (z^2)}{2! } \bar{z} \bigg)^2  \bigg\vert
   \leq (1 - \alpha)^2 ( 4 \alpha^2 - 12 \alpha + 13)
\end{align}
   and for $\alpha \in [0,2/3]$,
\begin{align}\label{crl4H3}
     \vert 2 d_2^2 d_3 & - d_3^2 - 2 d_2^2 + 1 \vert \leq 12 \alpha^4 -52 \alpha^3 + 91 \alpha^2 -74 \alpha + 24,
\end{align}
    where  $d_2$ and $d_3$ are given by (\ref{d2d3}). All these estimations are sharp.
\end{corollary}
\begin{remark}
     When $n=1$, (\ref{crl3H3}) and (\ref{crl4H3}) reduce to the bounds given in Theorem B.
\end{remark}
\begin{corollary}
 If $g: \mathbb{U}^n  \rightarrow \mathbb{C}$ and $G(z) = z g(z) \in \mathcal{SS}^*_\beta(\mathbb{U}^n)$, then for $\beta \in [1/3,1]$, the following sharp inequalities hold:
\begin{align}\label{crl51H3}
   \bigg\vert    \bigg( \frac{1}{\| z \|^4}  \frac{ D^3 G(0) (z^3)}{3! } \bar{z} \bigg)^2 -    \bigg(\frac{1}{\| z \|^3} \frac{D^2 G(0) (z^2)}{2! } \bar{z} \bigg)^2  \bigg\vert
   \leq 9 \beta^4 + 4 \beta^2
\end{align}
   and
\begin{align}\label{crl5H3}
     \vert 2 d_2^2 d_3 & - d_3^2 - 2 d_2^2 + 1 \vert \leq 15 \beta^4 + 8 \beta^2 + 1,
\end{align}
    where  $d_2$ and $d_3$ are given by (\ref{d2d3}).
\end{corollary}
\begin{remark}
     When $n=1$, (\ref{crl51H3}) and (\ref{crl5H3}) reduce to the bounds given in Theorem C.
\end{remark}
\section*{Declarations}
\subsection*{Funding}
The work of Surya Giri is supported by University Grant Commission, New Delhi, India  under UGC-Ref. No. 1112/(CSIR-UGC NET JUNE 2019).
\subsection*{Conflict of interest}
	The authors declare that they have no conflict of interest.
\subsection*{Author Contribution}
    Each author contributed equally to the research and preparation of the manuscript.
\subsection*{Data Availability} Not Applicable.
\noindent

\end{document}